\documentclass{amsart}

\usepackage{amssymb}
\usepackage{amsmath,amsfonts,amsthm,amstext}
\usepackage{amsmath}
\usepackage{mathrsfs}  
\usepackage[all]{xy}
\usepackage{xcolor}
\usepackage{tikz-cd}

\newtheorem{theorem}{Theorem}[section]
\newtheorem{lemma}[theorem]{Lemma}

\newtheorem{prop}[theorem]{Proposition}
\newtheorem{cor}[theorem]{Corollary}

\newcommand{\Z}{\mathbb{Z}}

\newcommand{\Q}{\mathbb{Q}}

\theoremstyle{definition}

\begin{document}

		\title[]{The free multiplicative Lie algebra $L(P)$ for a finitely generated parafree group $P$} \author{ 
  Dessislava H. Kochloukova}
	\address{State University of Campinas (UNICAMP), SP, Brazil \\
		\newline
		email : desi@unicamp.br
		} 
	\email{}
	%\subjclass[2010]{Primary 20J05; Secondary 20F05;}
	\date{}

\begin{abstract} Let $P$ be a group, $L(P)$ be the free multiplicative Lie algebra with normal subgroup $\Gamma_n(P)$ generated by Lie bracket ``commutators''  of weight $n$ and $\{ \gamma_n(P) \}$ be the lower central series of $P$. We prove  that $\Gamma_n(P) \simeq \gamma_n(P)$ for arbitrary $n \geq 1$ and $P$ a finitely generated parafree group such that $H_2(P, \mathbb{Z}) = 0 = H_3(P, \mathbb{Z})$ (e.g. $P$ satisfies the Strong Parafree Conjecture), in particular Ellis's conjecture holds i.e. the above isomorphism holds for finitely generated free group $P$ but this easily implies it holds for any free group. 
\end{abstract}

\maketitle

\section{Introduction}

Multiplicative Lie algebras were defined by Ellis in \cite{Ellis2}. Multiplicative Lie algebras can be viewed as generalisations of both Lie algebras and groups. The notion of non-abelian tensor product of multiplicative Lie algebras was introduced by Donadze, Inassaridze, Ladra in \cite{D-I-L} generalising the non-abelian tensor product of
groups defined by Dennis in \cite{Den} and  developed by Brown, Loday and Robertson  in \cite{B-L}, \cite{B-L1}, \cite{B-J-R} and generalising the  non-abelian tensor product of Lie algebras defined by Ellis in  \cite{Ellis3}. \iffalse{In \cite{D-I-L-V}  Donadze, Inassaridze, Ladra and Vieites  proved that the non-abelian tensor product of nilpotent, solvable and Engel multiplicative Lie algebras is
nilpotent, solvable and Engel, respectively. }\fi

In \cite{Ellis2} Ellis conjectured that five universal relations, stated in our preliminary section \ref{section2.1} as $(i) -(v)$,  applied to commutators of a fixed weight $n$ generate all universal relations between commutators of the same weight $n$.
Earlier  Miller proved in \cite{Miller}  that any universal
relation among commutators is deduced from four given relations without fixing the weights of the commutators.  The conjecture depends on the notion of 
a free multiplicative Lie algebra $L(P)$, which by definition is actually a group, defined   for any group $P$ in \cite{Ellis2} and some important normal subgroups $\Gamma_n(P)$ of $L(P)$. The identity group homomorphism on $P$ induces
a surjective map of multiplicative Lie algebras
$$\theta : L(P) \to P$$ where $P$ is viewed as a multiplicative Lie algebra with respect to the (standard) commutator, and $\theta$  restricts to a group epimorphism $$\theta_n: \Gamma_n(P) \to \gamma_n(P)$$ 
where $\{ \gamma_n(P) \}$ is the lower central series of $P$ given by $\gamma_1(P) = P$, $\gamma_{n+1} (P) = [\gamma_n(P), P]$ and the group $\Gamma_n(P)$ is the subgroup of $L(P)$
generated by the elements $\{\{\ldots \{\{x_1, x_2\}, x_3 \} , \ldots \} , x_n\}$ for $x_1, \ldots, x_n \in P$.  By \cite{Ellis2} if $G$ is  a $p$-group then $L(P)$ is  a $p$-group; if $P$ is finite then $\Gamma_n(P)$ is finite for all $n \geq 1$; if $P$  is nilpotent of class $c$ then each $\Gamma_n(P)$ is abelian for $n \geq c+1$. \iffalse{For perfect groups $P$ the group $L(P)$ is isomorphic to $U \rtimes P$, where $U$ is the universal central extension of $P$ i.e. $U \simeq P \wedge P$. }\fi

\medskip
{\bf The Ellis Conjecture for $L(P)$} {\it For a free group $P$ the epimorphism $\Gamma_n(P) \to \gamma_n(P)$ is  an isomorphism.}

\medskip
One of the main results of \cite{Ellis2} is that $\theta_2$ is an isomorphism if $H_2(P, \mathbb{Z}) = 0$ and $\theta_3$ is an isomorphism if $H_1(P, \Z)$ is torsion-free and $H_2(P, \mathbb{Z}) = 0$. The fact that $\theta_1$ is an isomorphism for an arbitrary group $P$ follows directly from the definition of $\theta_1$. The fact that for a free group $P$  the map $\theta_4$ is an isomorphism was proved by Donadze and Ladra in \cite{D-L1}. Surprisingly no new developments on the conjecture were published after \cite{D-L1}. Our first result is a positive solution of the conjecture.

\medskip
{\bf  Theorem A (= Weak Main Theorem)} {\it Let $F$ be a finitely generated free group. Then $\theta_n : \Gamma_n(F) \to \gamma_n(F)$ is an isomorphism for every $n \geq 1$.}

\medskip The case of infinitely generated free groups is solved later in Theorem B.
Our proof of the Theorem  A is  of homological nature, it uses the non-abelian exterior product $F \wedge \gamma_n(F)$ and in its final part it uses \cite[Remark2]{Ellis2}, which was the initial motivation for this work. But the similarities with the proofs of the low dimensional cases $n = 2,3,4$ end here.  The new ingredient is to use and  compare homologies of discrete groups with the pro-$p$ (i.e. continious)  homologies of the pro-$p$ completions for prime integer $p$. 
We prove that the map $H_3(G, \mathbb{Z}) \to H_3^{cont}( \widehat{G}_p, \mathbb{Z}_p)$, induced by the canonical map $G \to \widehat{G}_p$,  is injective for $G = F / \gamma_n(F)$ and $ \widehat{G}_p$ the pro-$p$ completion of $G$. 
This enables us to show  that $(F \wedge \gamma_n(F)) \rtimes F$ is residually $p$-finite, hence $\cap_m [F \wedge \gamma_n(F),_m F] = 1 $ and we can use \cite[Remark2]{Ellis2} that can be restated as $\cap_{m \geq 1} [\Gamma_n(F),_m F] = Ker (\Gamma_n(F) \to \gamma_n(F))$. In several places in the proofs we will use pro-p constructions that generalise the non-abelian tensor square of a pro-p group, first defined by Moravec in \cite{Mor}, but we will simply define them as the appropriate inverse limits of the corresponding finite discrete constructions. It is important that when the corresponding groups are $p$-finite the constructions $ \otimes, \wedge$ give $p$-finite groups  \cite{Ellis}.

We examine the conditions used in the proof of  Theorem A and we obtain Theorem \ref{cor}, a result where the group is not necessarily free but still it has to satisfy a lot of conditions due to the homological nature of our proof of   Theorem A.
We apply Theorem \ref{cor} for the class of finitely generated parafree groups introduced by Baumslag  in \cite{Baumslag}, \cite{Bau2}, \cite{Bau3}, \cite{Bau4}.  A group is parafree if it is residually nilpotent and its quotients by the terms of its lower central series are the same as those of a fixed free group. Every free group is parafree but  a parafree group is not necessarily free \cite{Ba-Cl}, \cite{Baumslag2}, \cite{JZ-M}.  In \cite{Coch} Cochran comments on the following  conjecture that is still open: 

\medskip
{\bf Strong Parafree Conjecture:} (Baumslag)  {\it Let $P$ be a
finitely generated parafree group. Then $H_2(P, \mathbb{Z}) = 0$ and the cohomological dimension $cd(P) \leq 2$.}

\medskip In the above conjecture the condition that the parafree group is finitely generated is not redundant.
Bousfield showed in \cite{Bous} that for a finitely generated, non-cyclic, free group $F$ and its the pro-nilpotent completion ${\mathcal{F}} = \varprojlim_n F / \gamma_n(F)$  the homology group $H_2(\mathcal{F}, \mathbb{Z})$ is not countable but $\mathcal{F}$ is  parafree. It is not known what is the cohomological dimension of $\mathcal{F}$ and whether $H_3(\mathcal{F}, \mathbb{Z})$ is zero. In \cite{B-I-M} Basok, Ivanov and Mikhailov proved that $H_2(\mathcal{F}, \mathbb{Z})$ is not a cotorsion group i.e. $Ext^1_{\Z}( \mathbb{Q},H_2(\mathcal{F}, \mathbb{Z})) \not= 0$. Recently Fisher and Klinge proved in \cite{Fish-Kl}  that if $P$ is a finitely generated
parafree group with $cd(P) = 2$, then $P$ satisfies the (Strong) Parafree Conjecture if
and only if the terms of its lower central series are eventually free. Furthermore in \cite[Remark 6.15]{Fish-Kl} it was shown that finitely generated parafree groups
satisfying the Strong Parafree Conjecture are virtually free-by-cyclic.

By Proposition \ref{parafree} every finitely generated parafree group $P$ such that  $H_2(P, \mathbb{Z}) = 0 = H_3(P, \mathbb{Z}) $ ( e.g. $P$  satisfies the Strong Parafree Conjecture) satisfies the conditions of Theorem \ref{cor}.

\medskip
{\bf Theorem B (= Strong Main Theorem)} {\it 
Let $P$ be a parafree group such that $H_2(P, \mathbb{Z}) = 0 = H_3(P, \mathbb{Z}) $ ( e.g. $P$ satisfies the conclusions of the Strong Parafree Conjecture). Suppose further that if $P$ is not finitely generated then it is free.  Then $\theta_n : \Gamma_n(P) \to \gamma_n(P)$ is an isomorphism for every $n \geq 1$.}

\medskip
We note that Stallings proved in \cite{Stal} that if $\varphi : H \to G$ is a group homomorphism   that induces an isomorphism $H/ \gamma_2(H) \to G/ \gamma_2(G)$ and an epimorphism $H_2(H, \Z) \to H_2(G, \Z)$ then 
$\varphi$ induces an isomorphism $H/ \gamma_n(H) \to G/ \gamma_n(G)$ for every $n \geq 1$. In particular if $H_2(G, \Z) = 0$, $G$ is residually nilpotent and $G/ \gamma_2(G)$ is free abelian group then $G$ is a parafree group.

\iffalse{ Next we characterise when  a finitely generated parafree group $P$ has trivial  Schur multiplier in terms of the residuall $p$-finiteness of the  group $\nu(P)$. The group $\nu(P)$ was defined by Rocco in \cite{Norai} and its generalisation $\eta(G, H) \simeq ((G \otimes H) \rtimes G) \rtimes H$ was defined by Ellis and Leonard in \cite{El-Le}, where  $\eta(G,G) \simeq \nu(G)$.

\medskip
{\bf Theorem C} {\it Let $P$ be a finitely generated parafree group and $p$ be an odd prime integer. Then $\nu(P)$ is residually $p$-finite if and only if $H_2(P, \Z) = 0$. }

\medskip We need in Theorem C that $p\not=2$ in order to apply a result of Blyth, Fumagalli and Morigi from \cite{B-F-M}  that if $G^{ab}$ does not contain 2-torsion then the epimorphism $G \otimes G \to G \wedge G$ between non-abelian tensor square and the non-abelian exterior square  splits.}\fi

Finally we observe that the Baumslag Strong Parafree Conjecture holds for 1-relator finitely generated parafree groups. We believe this  is well-known but as we could not find a reference we include a short proof.

\medskip
{\bf Corollary C} {\it Baumslag's Strong Parafree Conjecture holds for 1-relator finitely generated parafree groups $P$. In particular $\Gamma_n(P) \simeq \gamma_n(P)$ for all $ n $.
}

\medskip
{\bf Acknowledgements} The author is partially supported by bolsa de produtividade em pesquisa CNPq 301222/2026-6 and FAPESP 2024/14914-9. 
\section{Preliminaries} 

\subsection{Multiplicative Lie algebra} \label{section2.1} 
In \cite{Ellis2} the following universal commutator relations are considered: for every group $G$ and $x, x_1, y, y_1,z \in G$ we have

\medskip
(i) $[x,x] = 1$,

(ii) $ [x,y y_1] = [x,y ] \  (^y [x,y_1])$,

(iii) $[xx_1 ,y] = (^{x}[x_1, y]) \ [x, y]$ ,

(iv)  $[[y, x], ^x z][[x, z], ^z y][[z,y],^y x] = 1$,

(v) $^z[x,y] = [^z x,^z y]$

where $[x,y] = xy x^{-1}y^{-1}$ and $^x y = xy x^{-1}$. 

\medskip
In \cite{Ellis2} Ellis conjectured that these universal relations applied to commutators of a fixed weight $n$ generate all universal relations between commutators of  weight n. To formalise this Ellis introduced a new notion of multiplicative Lie algebra.
A multiplicative Lie algebra $D$ is  a multiplicative group $D$ with a function $\{ , \} : D \times D \to D$ 
satisfying the following identities for all $x, x_1, y, y_1,z \in D$:

\medskip
(i)' $\{x,x \} = 1$,

(ii)' $ \{x,y y_1\} = \{x,y \} (^y \{x,y_1\})$,

(iii)' $\{xx_1 ,y\} = (^{x}\{x_1, y\})\{x, y\}$ ,

(iv)'  $\{\{y, x\}, ^x z\}\{\{x, z\}, ^z y\}\{\{z,y\},^y x\} = 1$,

(v)' $^z\{x,y\} = \{^z x,^z y\}.$

\medskip
The above identities imply:

\medskip
(vi) $\{1, x\} = \{x, 1\} = 1$

(vii) $\{x,y\} = \{y,x\}^{-1}$ 

(viii) $^{ \{x,y \} } \{x_1,y_1\} = ^{[x, y]}\{x_1,y_1\}$

(ix) $[\{x,y\},x_1] = \{[x,y],x_1\}$

(x) $\{x^{-1},y\} = ^{x^{-1}}\{x,y\}^{-1}$ and $ \{y, x^{-1} \} = ^{x^{-1}} \{y, x\}^{-1}$

\medskip
Any group is a multiplicative Lie algebra with $\{ , \}$ coinciding with $[ , ]$.
Any Lie algebra over $\Z$ is a multiplicative Lie algebra with group operation $+$ and $\{ , \}$ the standard Lie bracket.

For any group $P$ there exists the free multiplicative Lie algebra $L(P)$ on $P$ such that
$P$ is a subgroup of $L(P)$ and for
any group homomorphism $P \to M$ from $P$ to a multiplicative Lie algebra
$M$ extends uniquely to a homomorphism of multiplicative Lie algebras $L(P) \to M$ i.e. that map commutes with group product and $ \{ -, - \}$. By definition $\Gamma_n(P)$ is the subgroup of $L(P)$
generated by the elements $\{\{\ldots \{\{x_1, x_2\}, x_3 \} , \ldots \} , x_n\}$ for $x_1, \ldots, x_n \in P$. Thus $\Gamma_1(P) = P$.  The group homomorphism
$$\theta_n : \Gamma_n(P) \to \gamma_n(P)$$ sends $\{\{\ldots \{\{x_1, x_2\}, x_3 \} , \ldots \} , x_n\}$ to the left-normed commutator $[x_1, x_2, x_3 , \ldots , x_n]$ By \cite{Ellis2} for any bracketing $\beta$ of the operation $\{ \ , \ \}$ of length $n$  we have $\beta(P) \subseteq \Gamma_n(P)$ and $\Gamma_n(P)$ is a normal subgroup of $L(P)$ for $n \geq 2$. Furthermore $L(P) = \cup_n \Gamma_1(P) \Gamma_2(P) \ldots \Gamma_n(P)$ and $\Gamma_i(P) \cap \Gamma_j(P) \not= 1$ for $i,j \geq 2$ since $1 \not= [\Gamma_i(P), \Gamma_j(P)] \subseteq \Gamma_i(P) \cap \Gamma_j(P)$.

\subsection{The non-abelian tensor product and the construction $ \eta$} \label{prelim-exact} Let $G$
and $H$ be groups acting compatibly on each other. We write $^g h $ for the element of $H$ produced by the action of $g \in G$ on $h \in H$ and $^h g $   for the element of $G$ produced by the action of $h \in H$ on $g \in G$. As usual $^{g_1} g = g_1 g g_1^{-1}$ and $^{h_1} h = h_1 h h^{-1}$ for $g, g_1 \in G, h_1, h \in H$. The compatibility condition on the action is $^{^g h} g_1 = ^{g h g^{-1}} g_1$ and $ ^{^h g} h_1 = ^{h g h^{-1}} h_1$ for all $g, g_1 \in G, h, h_1 \in H$.

The following construction was first defined in \cite{El-Le}, the reader can see
\cite{Irene}, \cite{Norai-essay}, \cite{Norai} too:  suppose $G$ and $H$ are groups and $\iota_1(G)$ is an isomorphisc copy of $G$ and $\iota_2(H)$ is an isomorphic copy of $H$. Then the group $\eta(G,H)$ is defined as follows:
$$\eta(G, H) = \langle\iota_1(G), \iota_2(H) \ | \ ^{\iota_1(g_1)}[\iota_1(g), \iota_2(h)] = [\iota_1(^{g_1} g) , \iota_2(^{g_1} h)],$$ $$ ^{\iota_2(h_1)}[\iota_1(g), \iota_2(h)]
 = [\iota_1(^{h_1} g), \iota_2(^{h_1}h)] \hbox{ for } g, g_1 \in G, h, h_1 \in H \rangle$$ 
By \cite[Prop. 1.4]{G-H} $[\iota_1(G), \iota_2(H)] \simeq G \otimes H$ the non-abelian tensor product, hence $$\eta(G, H) \simeq ((G \otimes H) \rtimes G) \rtimes H$$
The commutator $[i_1(g), i_2(h)]$ corresponds to $g \otimes h \in G \otimes H$.
By \cite{Ellis} if $G$ and $H$ are finite $p$-groups ( resp. finite groups)  then $G \otimes H$ is a finite $p$-group ( resp. a finite group), hence $\eta(G, H)$ is a finite $p$-group (resp. a finite group). 

\iffalse{ By definition $G \wedge G$ is the quotient of $G \otimes G$ by the subgroup $\Delta(G)$ generated by $g \otimes g$ for $g \in G$. We write $g_1 \wedge g_2$ for the image of $g_1 \otimes g_2$. Note that $\Delta(G)$ is a central subgroup of $G \otimes G$. 
We will need later in the proof of Theorem C several short exact sequences or corollaries of them. There is a (central) extension  
$$ 1 \to H_2(G, \Z) \to G \wedge G \to G' \to 1$$
that sends $g_1 \wedge g_2$ to $[g_1, g_2]$.

By \cite{B-L}, \cite{B-L1} there is an exact sequence
$$ H_3(G, \mathbb{Z}) \to  \Gamma(G^{ab}) \to  G \otimes G \to  G \wedge G$$
where the map $G \otimes G \to G \wedge G$ sends $g_1 \otimes g_2$ to $g_1 \wedge g_2$ and $\Gamma(-)$ is the Whitehead quadratic functor.

By \cite[Prop. 2.2 iii]{B-F-M}   for $A$  an abelian group without 2-torsion  $$\Gamma(A) = Ker (A \otimes A \to A \wedge A)$$ 
By \cite{B-F-M} if $G^{ab}$ is finitely generated and does not have 2-torsion then the epimorphism $G \otimes G \to G \wedge G$ splits.}\fi

Suppose $P$ is a group,  $M$ and $N$ normal subgroups of $P$ such that $P = MN$. Then $M$ and $N$ act on each other via conjugation in $P$. Since  these actions are compatible we can consider the non-abelian exterior product $M \wedge N$. By definition $M \wedge N$ is the quotient of the non-abelian tensor product $M \otimes N$ by the subgroup generated by the elements $m \otimes m$ where $m \in M \cap N$. By \cite[Thm. 4.5]{B-L} or \cite[Cor. 2]{B-L1} there is an exact sequence ( we omit some low dimensional factors) 
$$H_3(P, \Z) \to H_3(P/M, \Z) \oplus H_3(P/N,\Z) \to V \to H_2(P, \Z) $$
where $V = Ker(M \wedge N \to P)$ with $M \wedge N \to P$ sending $m \wedge n$ to $[m,n]$.

\subsection{Third integral homology of a free nilpotent group}
In \cite{Ig-Orr} Igusa and Orr calculate the homology groups of free nilpotent groups $G = F / \gamma_{n}(F)$  by comparing  them with the homology groups of the Lie algebra $L$ (over $\Z$) defined by the lower central series $\{ \gamma_i(G) \}$ i.e. $L = \oplus_{1 \leq i \leq n-1}  \gamma_i(G)/ \gamma_{i+1} (G) $. By \cite[Cor. 6.5]{Ig-Orr} $H_3(G, \Z) \simeq H_3(L, \Z)$ and  the filtration of the third homology given by the spectral sequence used (the May spectral sequence) has quotients that are free abelian groups and the ranks are calculated. In particular $$H_3(G, \Z) \hbox{ is torsion-free for } G =  F / \gamma_{n}(F)$$ We note that the situation changes radically when we move to the fourth homology i.e. $H_4(G, \Z)$ could contain torsion, see \cite{Rom}.

\subsection{Profinite homology} Let $p \geq 2$ be a prime integer. For a discrete groups $G$ we denote by $\widehat{G}_p$ the pro-$p$ completion of $G$ i.e. the inverse limit of all finite $p$-group quotients of $G$. The algebra of $p$-adic integers $\mathbb{Z}_p$ is the pro-$p$ completion of $\mathbb{Z}$. We will work in this paper with both homology of discrete groups $H_i(G, - )$ and the pro-$p$ version we denote by $H_n^{cont}(\widehat{G}_p, -)$. We refer the reader to \cite{book} for more details on continious (profinite) homology. Here we will use that the profinite homology  $H_n^{cont}$ commutes with inverse limits, hence if $\varprojlim_i G_i = G_0$, $G_i$ finite group quotients ( in our case $p$-groups) of the profinite group $G_0$ ( in our case pro-$p$ group) and  $\varprojlim_j M_j = M_0$ where $M_0$ is a pro-finite ( in our case pro-$p$) $G_0$-module where the action of $G_0$ on $M_j$ factors through $G_i$ and each $M_i$ is finite (in our case $p$-finite) then by basic properties of profinite homology, see \cite{book},
$$H_n^{cont}(G_0, M_0) \simeq  \varprojlim_{j,i} H_n^{cont}(G_i, M_j) \simeq \varprojlim_{j,i} H_n(G_i, M_j)$$
where the last isomorphism comes from the fact that abstract and continious (profinite) homology coincide for finite groups $G_j$ with finite coefficients modules $M_j$.
Furthermore by properties of homology of discrete groups, see \cite{Rotman},
if $|G_j| = m_j$  and $M$ is a discrete $G_j$-module we have
$$m_j H_n(G_j, M) = 0$$

\subsection{Parafree groups} We recall that a group $P$ is parafree if it is residually nilpotent and its quotients by the terms of its lower central series are the same as those of a free group $F$ i.e. this means
that there is an isomorphism $$F / \gamma_n(F) \to P/ \gamma_n(P)$$ for each positive integer $n$, and
that these isomorphisms are compatible with each other in the obvious way.

In \cite{Mag} Magnus classified the finitely generated free groups in the class of parafree groups: 
if $P$ is an $n$–generator group ($n < \infty$) and $P$
has the same lower central sequence as a free group of rank $n$ then $P$ is free. 
In \cite{JZ-M} Jaikin-Zapirian and Morales  considered parafree groups that split as fundamental groups of graph of groups with cyclic edge groups and vertex groups that are parafree. One of the starting points of the considerations in \cite{JZ-M} is the following proposition whose proof is based on results of Lackenby \cite{Mark} and Gruenberg  \cite{Gruenberg}.

\begin{prop} \cite[Prop. 2.1]{JZ-M} \label{Jaikin}  Let $P$ be finitely generated residually nilpotent group. Then, $P$ is parafree if and
only if $\widehat{P}_p$  is a free pro-$p$ group for every prime $p$. Moreover, a parafree group is residually-$p$ (= residually $p$-finite)  for every
prime $p$. \end{prop}

 Here we collect some  examples of parafree groups due to Baumslag \cite{Baumslag2}, the groups are given by generators and relations:

1) $\langle a,b,c \ | \ a^2 b^p c^p \rangle$ for odd prime $p$,

2)  $\langle a,b,c \ | \ a=[c^i, a][c^j,b]  \rangle$ for $i,j $ positive integers,

3)  $\langle a,b,c \ | \ a=[a^i, b^j][c,b]  \rangle$ for $i,j $ positive integers,

4)  $\langle a,b,c \ | \ a^i[b, a]= c^j  \rangle$ for $i,j $ positive integers.

\section{Proof of Theorem A}

Let $F$ be a finitely generated free group.
Consider $F$ and $\gamma_n(F)$ with mutual actions given by conjugation in $F$.
Consider the discrete group $\eta(F, \gamma_n(F))$.
Set $$\Delta (F) = \langle g \otimes g \ | \ g \in \gamma_n(F) \rangle \leq \eta(F, \gamma_n(F))$$ Then $$F \otimes \gamma_n(F) / \Delta(F) \simeq F \wedge \gamma_n(F)$$ the non-abelian exterior product. By \cite[Cor. 2]{B-L1}, \cite{Ellis1} there is a short exact sequence
$$1 \to H_3(F/ \gamma_n(F), \mathbb{Z}) \to F \wedge \gamma_n(F) \to \gamma_{n+1}(F) \to 1$$ where the map $$F \wedge \gamma_n(F) \to \gamma_{n+1}(F)$$ sends $g_1 \wedge g_2$ to $[g_1, g_2]$.

We write $\iota_1(F)$ for the copy of $F$ in $\eta(F, \gamma_n(F))$ and $\iota_2(\gamma_n(F))$ for the copy of $\gamma_n(F)$ in $\eta(F, \gamma_n(F))$. Then 
$$[\iota_1(F), \iota_2(\gamma_n(F))] \simeq F \otimes \gamma_n(F)$$

Define  $\mu(F, \gamma_n(F))$ as the subgroup of $\eta(F, \gamma_n(F))$ generated by the groups  $\iota_1(F)$ and $[\iota_1(F), \iota_2(\gamma_n(F))]$.  Note that $$\mu(F, \gamma_n(F))/ \Delta(F) \simeq (F \wedge \gamma_n(F))  \rtimes F$$
The same construction could be extended to any group $G$ i.e.  $\mu(G, \gamma_n(G))$ is the subgroup of $\eta(G, \gamma_n(G))$ generated by the groups  $\iota_1(G)$ and $[\iota_1(G), \iota_2(\gamma_n(G))]$.

Let $U$ be a $p$-finite quotient of $F$. Consider $$(U \wedge \gamma_n(U)) \rtimes U \simeq \mu(U, \gamma_n(U))/ \Delta(U) \leq \eta(U, \gamma_n(U))/ \Delta(U) $$
Since $U$ is a finite $p$-group,  $ \eta(U, \gamma_n(U))$ is a finite $p$-group, hence 
$(U \wedge \gamma_n(U)) \rtimes U $
is a finite $p$-group. 

Let $G = F/ \gamma_n(F)$. Set $\overline{\gamma_n}: = \overline{\gamma_n(\widehat{F}_p)}$, where $ \widehat{F}_p$ is the pro-$p$ completion of $F$ and we use overlining for closure in the pro-$p$ group $\widehat{F}_p$. Note that $\widehat{F}_p/ \overline{\gamma_n}$ is the maximal pro-$p$ quotient of $\widehat{F}_p$ that is nilpotent of nilpotency class $n-1$. Define 
$$\widehat{F}_p \widehat{\wedge} \overline{\gamma_n} := \varprojlim_U U \wedge \gamma_n(U)$$
where the inverse limit is over all $p$-finite group quotients $U$ of $F$, thus
$$\widehat{F}_p \simeq \varprojlim_U U$$
There is a homomorphism of pro-$p$ groups
$$\widehat{F}_p \widehat{\wedge} \overline{\gamma_n} \to \overline{\gamma_{n+1}}$$
induced by the homomorphism
$$F/ U \wedge \gamma_n(F/U) \to \gamma_{n+1}(F/U)$$
that sends $a \wedge b$  to  $[a,b]$ for $a \in F/U, b \in \gamma_n(F/U)$. 

\begin{lemma} \label{L1} Let $p \geq 2$ be a prime integer. Then ${\mathcal{M}}: = Ker(\widehat{F}_p \widehat{\wedge} \overline{\gamma_n} \to \overline{\gamma_{n+1}})$ is a pro-$p$ subgroup of $H_3^{cont}( \widehat{G}_p, \mathbb{Z}_p)$.
\end{lemma}

\begin{proof} 

Let $K$ be a normal subgroup of $F$ such that $U = F/ K$ is a finite $p$-group and denote $R = \gamma_n(F)$. Consider the last exact sequence from section \ref{prelim-exact} for $P = M = F/ K, N = RK/ K$
$$H_3(F/K, \Z) \to  H_3(F/ RK,\Z) \to V_K \to H_2(F/K, \Z) $$
where $V_K : = Ker(F/K \wedge RK/K \to F/K)$ with $F/K \wedge RK/K \to F/K $ sending $m \wedge n$ to $[m,n]$ for $m \in F/K, n \in RK/K$. All groups in the above exact sequence are finite $p$-groups, but since $ \varprojlim^1$ applied to a tower of finite groups is trivial i.e. zero, applying inverse limit over all normal subgroups $K$ of $F$ with $F/ K$ $p$-finite gives an exact sequence
\begin{equation} \label{exact1}  \varprojlim_K H_3(F/K, \Z) \to \varprojlim_K H_3(F/ RK,\Z) \to \varprojlim_K V_K \to \varprojlim_K H_2(F/K, \Z) \end{equation}

Let $m_K =   | F/ K| = p^{i_K}$ and $n_K =  |F/ RK| = p^{j_K}$,  hence $m_KH_i(F/K, \mathbb{Z}) = 0 =  n_k H_3(F/ RK, \mathbb{Z})$ for $i \geq 1$. Then by the universal coefficients theorem there are exact sequences of abelian groups
$$0 \to H_s(F/K, \mathbb{Z}) \to H_s(F/K, \mathbb{Z}/ p^{i} \mathbb{Z}) \hbox{ for } i \geq i_k, s \in \{ 2,3 \}$$
and
$$0 \to H_3(F/ RK, \mathbb{Z}) \to H_3(F/ RK, \mathbb{Z}/ p^{j} \mathbb{Z}) \hbox{ for } j \geq j_k$$
Since inverse limit is left exact functor applying inverse limit over all normal subgroups $K$ of $F$ with $F/ K$ $p$-finite gives  exact sequences
$$0 \to \varprojlim_K  H_s(F/K, \mathbb{Z}) \to \varprojlim_{K, i \geq i_K}  H_s(F/K, \mathbb{Z}/ p^i \mathbb{Z}) \hbox{ for } s \in \{ 2,3 \}$$
and
$$0 \to \varprojlim_K  H_3(F/ RK, \mathbb{Z}) \to  \varprojlim_{K, j \geq j_K}  H_3(F/ RK, \mathbb{Z}/ p^j \mathbb{Z})$$
Since $\widehat{F}_p$ is a free pro-$p$ group  
$$\varprojlim_{K, i \geq i_k}  H_s(F/K, \mathbb{Z}/ p^i \mathbb{Z})  \simeq H_s^{cont}(\varprojlim_K  F/K,\varprojlim_{i \geq i_k}  \mathbb{Z}/ p^i \mathbb{Z}) \simeq  H_s^{cont} (\widehat{F}_p, {\mathbb{Z}}_p) = 0 \hbox{ for } s \in \{ 2,3 \}$$
Hence
$$ \varprojlim_K  H_s(F/K, \mathbb{Z})  = 0 \hbox{ for } s \in \{ 2,3 \}$$
 Recall  $F/ R = G$, hence $\varprojlim_K F/ RK \simeq \widehat{G}_p$, hence
$$\varprojlim_{K, j \geq j_K} H_3(F/ RK, \mathbb{Z}/ p^j \mathbb{Z}) \simeq  H_3^{cont}( \varprojlim_K F/ RK, \varprojlim_{j \geq j_k} \mathbb{Z}/ p^j \mathbb{Z}) \simeq H_3^{cont}(\widehat{G}_p, \mathbb{Z}_p)$$
Combining with (\ref{exact1}) we get an isomorphism
$$\varprojlim_K V_K \simeq \varprojlim_K H_3(F/ RK,\Z)$$
and $$ \varprojlim_{K} H_3(F/ RK,\Z)\hbox{ is a subgroup of } \varprojlim_{K, j \geq j_k}  H_3(F/ RK, \mathbb{Z}/ p^j \mathbb{Z}) \simeq H_3^{cont}(\widehat{G}_p, \mathbb{Z}_p)$$

Finally consider the short exact sequence of  groups
$$1 \to V_K \to F/K \wedge RK/K \to [F, R]K/K \to 1$$
for $F/K$ finite.
Since all  the groups in the exact sequence are finite by applying inverse limit over all normal subgroups $K$ of $F$ such that $F/ K$ is  $p$-finite i.e. a finite $p$-group, we get an exact sequence
$$
1 \to  \varprojlim_K V_K\to  \varprojlim_K  F/K \wedge RK/K  \to  \varprojlim_K [F, R]K/K \to 1$$
but actually we only need that $
1 \to  \varprojlim_K V_K\to  \varprojlim_K  F/K \wedge RK/K  \to  \varprojlim_K [F, R]K/K $ is exact and this follow from the fact that inverse limit is left exact.
We identify $\widehat{F}_p \widehat{\wedge} \overline{\gamma_n}$ with $ \varprojlim_K  F/K \wedge RK/K$ and $\overline{\gamma_{n+1}}$ with $\varprojlim_K [F, R]K/K $, so
$$ \varprojlim_K V_K \simeq Ker ( \widehat{F}_p \widehat{\wedge} \overline{\gamma_n} \to \overline{\gamma_{n+1}})$$
\end{proof}

\begin{lemma} \label{L4} Let $p \geq 2$ be a prime integer.
Let $ 1 \to C \to G \to K \to 1$ be a short exact sequence of groups, where $C$ is finitely generated abelian and $G$ is residually $p$-finite. Then there is an induced short exact sequence $1 \to \widehat{C}_p \to \widehat{G}_p \to \widehat{K}_p \to 1$.
\end{lemma}

\begin{proof}
Note that the pro-$p$ completion of discrete groups is a right exact functor, so it remains to prove that the pro-$p$ topology of $G$ induces on $C$  the pro-$p$ topology of $C$. Let $\{ C_i \}$ be subgroups of $C$ that define the topology of $C$ induced by the pro-p topology of $G$. Since $G$ is residually $p$-finite $\cap_i C_i = 1$. For every $i$ there is  a $p$-power $p^{k_i}$ such that $C^{p^{k_i}} \subseteq C_i$, hence $\cap_i C^{p^{k_i}} = 1$. This together with $C$ is finitely generated abelian implies that  $\{ C_i \}$ defines the pro-$p$ topology of $C$ i.e. $\widehat{C}_p \simeq \varprojlim_i C/ C_i$. 
\end{proof}
\begin{lemma}  \label{BK1}   Let $G$ be  a discrete group with $H_3(G, \mathbb{Z})$ finitely generated, $p$ be a prime integer and  the map $H_3(G, \mathbb{Z}/ p^j \mathbb{Z}) \to H_3^{cont}(\widehat{G}_p, \mathbb{Z}/ p^j \mathbb{Z})$, induced by the canonical map $G \to \widehat{G}_p$, be an isomorphism  for every integer $j \geq 1$. Then $Ker(H_3(G, \mathbb{Z}) \to H_3^{cont}(\widehat{G}_p, {\mathbb{Z}_p}))$ is a   finite group of order coprime to $p$.
\end{lemma}

\begin{proof}   By the universal coeficients theorem there is an exact sequence
    $0 \to $ $ H_3(G, \mathbb{Z}) \otimes_{\mathbb{Z}} \mathbb{Z}/ p^j \mathbb{Z} $ $ \to H_3(G,  \mathbb{Z}/ p^j \mathbb{Z}) $. Taking inverse limit over all $j \geq 1$ we get an exact sequence
    $$\label{eq111} 0 \to \varprojlim_j H_3(G, \mathbb{Z}) \otimes_{\mathbb{Z}} \mathbb{Z}/ p^j \mathbb{Z} \to \varprojlim_j H_3(G,  \mathbb{Z}/ p^j \mathbb{Z}) $$
Note that
    $$ \label{eq222}  \varprojlim_j H_3(G,  \mathbb{Z}/ p^j \mathbb{Z}) \simeq \varprojlim_j H_3^{cont}(\widehat{G}_p,  \mathbb{Z}/ p^j \mathbb{Z}) \simeq H_3^{cont}(\widehat{G}_p, \varprojlim_j \mathbb{Z}/ p^j \mathbb{Z}) \simeq H_3^{cont}(\widehat{G}_p,{\mathbb{Z}}_p). 
    $$
  The map
   $ H_3(G, \mathbb{Z}) \to \varprojlim_j H_3(G, \mathbb{Z}) \otimes_{\mathbb{Z}} \mathbb{Z}/ p^j \mathbb{Z}$ has kernel $ \cap_j p^j H_3(G, \mathbb{Z})$, hence 
   $$ \cap_j p^j H_3(G, \mathbb{Z}) = Ker(H_3(G, \mathbb{Z}) \to H_3^{cont}(\widehat{G}_p, {\mathbb{Z}_p}))$$
   Finally since $H_3(G, \mathbb{Z}) $ is a finitely generated abelian group, hence a direct sum of cyclic group,   
   $\cap_j p^j H_3(G, \mathbb{Z})$ is a $p'$-finite group. 
 \end{proof}

\begin{lemma} \label{L2}  Let $G = F/ \gamma_n(F)$ be the free nilpotent group of class $n-1$ and  $p \geq 2$ be a prime integer. Then the  map $H_3(G, \mathbb{Z}) \to H_3^{cont}(\widehat{G}_p, \mathbb{Z}_p)$, induced by the canonical map $G \to \widehat{G}_p$, is injective. \end{lemma}

\begin{proof} By Lemma \ref{BK1} together with the fact that $H_3(G, \mathbb{Z})$ is torsion-free \cite{Ig-Orr}, it suffices to show that \begin{equation} \label{p-good} H_3(G,  \mathbb{Z}/ p^j \mathbb{Z}) \to H_3^{cont}(\widehat{G}_p,  \mathbb{Z}/ p^j \mathbb{Z})\end{equation} is an isomorphism, where $H_3^{cont}(\widehat{G}_p,  \mathbb{Z}/ p^j \mathbb{Z})$ is the pro-p homology. Note that (\ref{p-good})  is not an isomorphism if $G$ is infinitely generated free countable nilpotent group, since in this case $H_3(G,  \mathbb{Z}/ p^j \mathbb{Z}) $ is countable and $H_3^{cont}(\widehat{G}_p,  \mathbb{Z}/ p^j \mathbb{Z})$ is not countable.

We can use induction on the nilpotency class of $G$ to show that $H_3(G,  \mathbb{Z}/ p^j \mathbb{Z}) \to H_3^{cont}(\widehat{G}_p,  \mathbb{Z}/ p^j \mathbb{Z})$ is an isomorphism for all $i$. The base of the induction is when $G$ is  finitely generated, free abelian i.e.  $G \simeq \mathbb{Z}^s$, then   $$\mathbb{Z}^{{s}\choose{3}} \simeq \wedge^3 G \simeq H_3(G,  \mathbb{Z}) \hbox{ and   }{\mathbb{Z}_p}^{{s}\choose{3}} \simeq \widehat{\wedge} ^3 \widehat{G}_p \simeq  H_3^{cont}(\widehat{G}_p, \mathbb{Z}_p)$$ and the map $$\mathbb{Z}^{{s}\choose{3}} \simeq H_3(G, \mathbb{Z}) \to {\mathbb{Z}_p}^{{s}\choose{3}} \simeq H_3^{cont}(\widehat{G}_p, \mathbb{Z}_p)$$ is actually $- \otimes_{\mathbb{Z}} {\mathbb{Z}}_p$.  Then the map $$H_3(G, \mathbb{Z}) \otimes_{\mathbb{Z}}  \mathbb{Z}/ p^j \mathbb{Z} \to H_3^{cont}(\widehat{G}_p, \mathbb{Z}_p) \widehat{\otimes}_{\mathbb{Z}_p}  \mathbb{Z}/ p^j \mathbb{Z}$$ induced by $G \to \widehat{G}_p$, is an isomorphism.
By the universal coefficients theorem there are natural exact sequences
$$0 \to H_3(G,  \mathbb{Z}) \otimes_{\mathbb{Z}}  \mathbb{Z}/ p^j \mathbb{Z} \to H_3(G,  \mathbb{Z}/ p^j \mathbb{Z}) \to Tor^{\mathbb{Z}}_1(H_{2}(G, \mathbb{Z}),  \mathbb{Z}/ p^j \mathbb{Z}) \to 0$$ and 
$$0 \to H_3^{cont}(\widehat{G}_p, \mathbb{Z}_p) \widehat{\otimes}_{\mathbb{Z}_p}  \mathbb{Z}/ p^j \mathbb{Z} \to H_3^{cont}(\widehat{G}_p,  \mathbb{Z}/ p^j \mathbb{Z}) \to Tor^{\mathbb{Z}_p}_1(H_{2}^{cont}(\widehat{G}_p, \mathbb{Z}_p),  \mathbb{Z}/ p^j \mathbb{Z}) \to 0$$
Some explanation is required why we can apply the universal coefficients theorem for the pro-$p$ group $\widehat{G_p}$. The point is that \iffalse{$\widehat{G}_p$ is of homological type $FP_{\infty}$, so all homology groups
$ H_s^{cont}(\widehat{G}_p, \mathbb{Z}_p)$ are finitely generated $\Z_p$-modules}\fi $\mathbb{Z}_p$ is a PID and $M \widehat{\otimes}_{\Z_p} (\Z/ p^j \Z) \simeq M / p^j M \simeq M \otimes_{\Z_p} (\Z/ p^j \Z)$ for every abelian profinite group $M$.

Since $H_{2}(G, \mathbb{Z})$ is a finitely generated torsion-free abelian group, it is a free $\mathbb{Z}$-module, so $Tor^{\mathbb{Z}}_1(H_{2}(G, \mathbb{Z}),  \mathbb{Z}/ p^j \mathbb{Z}) = 0$. Similarly $Tor^{\mathbb{Z}_p}_1(H_{2}(\widehat{G}_p, \mathbb{Z}_p),  \mathbb{Z}/ p^j \mathbb{Z}) = 0$. This completes the  base of the induction.

There is a central extension $$ 1 \to C \to G \to K \to 1$$ where $G = F/ \gamma_n(F)$, $C = \gamma_{n-1}(F)/ \gamma_n(F)$. Since $F$ is finitely generated, $C$ is a finitely generated abelian group and we can apply
 Lemma \ref{L4}, so there is a central extension $$ 1 \to \widehat{C}_p \to \widehat{G}_p \to \widehat{K}_p \to 1$$

Consider the Lyndon-Hochschild-Serre spectral sequence $$E_{i,k}^2 = H_i(K, H_k(C,  \mathbb{Z}/ p^j \mathbb{Z})) $$ that converges to $H_{i+k} (G,  \mathbb{Z}/ p^j \mathbb{Z})$. Note that $C$ is a central subgroup of $G$, hence $K$ acts trivially on $H_k(C,  \mathbb{Z}/ p^j \mathbb{Z})$. 

Consider the pro-p 
 Lyndon-Hochschild-Serre spectral sequence $$\widetilde{E}_{i,k}^2 =  H_i^{cont}(\widehat{K}_p, H_k^{cont}(\widehat{C}_p,  \mathbb{Z}/ p^j \mathbb{Z}))$$ that converges to $H_{i+k}^{cont} (\widehat{G}_p,  \mathbb{Z}/ p^j \mathbb{Z})$. 
 Note that $\widehat{C}_p$ is a central subgroup of $\widehat{G}_p$, hence $\widehat{K}_p$ acts trivially on $H_k(\widehat{C}_p,  \mathbb{Z}/ p^j \mathbb{Z})$. 
 
  By induction the canonical map $H_i(K,  \mathbb{Z}/ p^s \mathbb{Z}) \to  H_i^{cont}(\widehat{K}_p,  \mathbb{Z}/ p^s \mathbb{Z})$ is an isomorphism for all $s \geq 1$ and  $H_k(C,  \mathbb{Z}/ p^j \mathbb{Z}) \to H_k^{cont}(\widehat{C}_p,  \mathbb{Z}/ p^j \mathbb{Z})$  is an isomorphism between $p$-finite abelian groups, so direct sums of copies of some $\mathbb{Z}/ p^s \mathbb{Z}$, where repetition is permitted. Hence the 2-pages of the spectral sequences are naturally isomorphic, hence the infinite pages are naturally isomorphic and the map  $H_n(G,  \mathbb{Z}/ p^j \mathbb{Z}) \to H_n^{cont}(\widehat{G}_p,  \mathbb{Z}/ p^j \mathbb{Z})$ is an isomorphism for all $n, j \geq 1$.
\end{proof}

\begin{lemma} \label{L3} Let $p \geq 2$ be a prime integer. The canonical map $F \wedge \gamma_n(F) \to \widehat{F}_p \widehat{\wedge} \overline{\gamma_n}$ is injective, hence the canonical map
$$\mu(F, \gamma_n(F))/ \Delta \simeq (F \wedge \gamma_n(F)) \rtimes F \to  \varprojlim_U (U \wedge \gamma_n(U)) \rtimes U$$ is injective, where the inverse limit is over all $p$-finite ( i.e. finite $p$-group) quotients $U$ of $F$. In particular $(F \wedge \gamma_n(F)) \rtimes F$ is residually $p$-finite, hence residually nilpotent.
\end{lemma}
\begin{proof}
Consider the commutative diagram with exact rows
\[ 
\begin{tikzcd}
 1\arrow{r}{}&  H_3(G, \mathbb{Z})  \arrow{r}
\arrow{d}{\alpha}
 &   F \wedge \gamma_n(F) \arrow{r}{} \arrow{d}{\beta}& \ \gamma_{n+1}(F) \arrow{r}{}\arrow{d}{\delta}&  1\\%
 1 \arrow{r}{}& \mathcal{M} \arrow{r}{}  &  \widehat{F}_p \widehat{\wedge} \overline{\gamma_n}  \arrow{r}{} & \ \overline{\gamma_{n+1}} \arrow{r}{}&  1\\%
\end{tikzcd}
\]
where $\beta$ is induced by the canonical maps $F \to \widehat{F}_p$ and $\gamma_n(F) \to \overline{\gamma_n}$.
The map $\alpha$ composed with the embedding of $\mathcal{M}$ in $H_3( \widehat{G}_p, \mathbb{Z}_p)$ given by Lemma \ref{L1}  gives the  map $H_3(G, \mathbb{Z}) \to H_3( \widehat{G}_p, \mathbb{Z}_p)$ that is injective by Lemma \ref{L2}, hence $\alpha$ is injective. Recall that $\overline{\gamma_{n+1}}$ is the closure of $\gamma_{n+1}(F)$ in $ \widehat{F}_p$,  $\delta$ is a restriction of the canonical map $F \to \widehat{F}_p$, that is injective since $F$ is residually $p$-finite, hence $\delta$ is injective. Since both $\alpha$ and $\delta$ are injective we conclude that $\beta$ is injective.

Finally since $F \wedge \gamma_n(F) \to \widehat{F}_p \widehat{\wedge} \overline{\gamma_n}$ is injective and $F \to \widehat{F}_p$ is injective ( i.e. $F$ is residually $p$-finite) we deduce that $\mu(F, \gamma_n(F))/ \Delta \simeq (F \wedge \gamma_n(F)) \rtimes F \to  \varprojlim_U (U \wedge \gamma_n(U)) \rtimes U$ is injective.
\end{proof}

\begin{theorem} \label{Main}  (= Theorem A) Let $F$ be a finitely generated free group. Then the surjective map $\theta_n: \Gamma_n(F) \to \gamma_n(F)$ is an isomorphism for every $n \geq 1$. \end{theorem}

\begin{proof} We use induction on $n$. The cases $n = 2$ and $n =3$ are done in \cite{Ellis2} and $n = 4$ in \cite{D-L1}.

Suppose that the map $\theta_n: \Gamma_n(F) \to \gamma_n(F)$ is an isomorphism. Then there are surjective maps $F \wedge \gamma_n(F) \to \gamma_{n+1} (F)$ and $\theta_{n+1} : \Gamma_{n+1}(F) \to \gamma_{n+1} (F)$, where the first map sends $f \wedge w$ to $[f,w]$. 

By the definition of  the groups $\Gamma_i(F)$, $i \geq 1$, there is an epimorphism of groups $\rho_n : F \wedge \Gamma_n(F) \to \Gamma_{n+1}(F)$ that sends $a \wedge b$ to $ \{ a, b \}$ .  Then we have a homomorphism $$ \theta_{n+1}  \rho_n (id_F \wedge  \theta_n^{-1}):  F \wedge \gamma_n(F) \to \gamma_{n+1}(F)$$ that sends $f \wedge c$ to $[f,c]$, hence 
$$A: = Ker(  \theta_{n+1}  \rho_n (id_F \wedge  \theta_n^{-1})) \simeq H_3(G, \mathbb{Z}), \hbox{ where } G = F/ \gamma_n(F).$$
The above isomorphism is a particular case of the last exact sequence in section \ref{prelim-exact} for $P = M = F, N = \gamma_n(F)$.

Since the map $ \theta_{n+1}  \rho_n (id_F \wedge  \theta_n^{-1})$ factors through $\Gamma_{n+1} (F)$ we conclude that 
 $$\Gamma_{n+1} (F) \simeq F \wedge \gamma_n(F)/ B$$ where $B$ is a subgroup of $A$.

 If $B = A$ then $\theta_{n+1} : \Gamma_{n+1}(F) \to \gamma_{n+1} (F)$ is an isomorphism. 

Suppose $B \not= A$ and let $p$ be a prime such that if $A/ B$ is not torsion-free then $p$ divides the order of the torsion part of $A/ B$. If $A/ B$ is torsion-free $p$ can be any prime integer.

Since  $(F \wedge \gamma_n(F)) \rtimes F$ is residually $p$-finite, it is residually nilpotent. Hence  $$\cap_m [F \wedge \gamma_n(F) , _m F] = 1$$ where $[C,_m D]$ is the left-normed commutator $[C, D, \ldots, D]$ where $D$ appears $m$ times. Indeed since $(F \wedge \gamma_n(F)) \rtimes F$ is residually nilpotent we have $\cap_{m \geq 1} \gamma_m (((F \wedge \gamma_n(F)) \rtimes F) = 1$ and $[F \wedge \gamma_n(F) , _m F] \subseteq \gamma_{m+1}((F \wedge \gamma_n(F)) \rtimes F)$.

Define $$A_m: = [F \wedge \gamma_n(F) , _m F] \cap A$$ and note that 
$$ ([(F \wedge \gamma_n(F)), _m F] B) \cap A  = A_m B $$  We claim that  \begin{equation} \label{eq-eq2} \cap_m ([(F \wedge \gamma_n(F)), _m F] B) \cap A  \subseteq \cap_m A_m B \subseteq  \widetilde{B}\end{equation}
where $\widetilde{B}$ is the intersection of the closure of $B$  (under the the pro-p topology of $\mu(F, \gamma_n(F))$) with $A$, the closure of $B$ is $\cap_V BV$ where the intersection is over all normal subgroups $V$ of $(F \wedge \gamma_n(F)) \rtimes F$ of $p$-power index.
  Indeed $$\widetilde{B} = (\cap_V BV) \cap A = \cap_V (BV \cap A) = \cap_V B(V \cap A) $$  If $(F \wedge \gamma_n(F)) \rtimes F/ V$ is nilpotent of class $m$, then 
 $[F \wedge \gamma_n(F) , _m F] \subseteq V$, so $$[F \wedge \gamma_n(F) , _m F] \cap A = A_m \subseteq V \cap A \subseteq B(V \cap A) \hbox{ and  }A_m B \subseteq B(V \cap A)$$ Hence $$\cap_m A_m B \subseteq \cap_V B(V \cap A) = \widetilde{B}$$

Note that since $G$ is a finitely generated free nilpotent group, it is polycyclic, hence of homological type $FP_{\infty}$, hence all homology groups $H_i(G, \Z)$ are finitely generated, in particular  $$A \simeq H_3(G, \mathbb{Z}) \hbox{ is a finitely generated abelian group}$$ We already proved that $(F \wedge \gamma_n(F)) \rtimes F$ is residually $p$-finite, hence $(F \wedge \gamma_n(F)) \rtimes F$  embeds in its pro-$p$ completion. Then $A$ embeds in the pro-$p$ completion of $ (F \wedge \gamma_n(F)) \rtimes F$.

By Lemma \ref{L4} the pro-$p$ topology of $(F \wedge \gamma_n(F)) \rtimes F$ induces the  pro-$p$ topology of $A$.
 Then since $A$ is finitely generated abelian, if $B$ is a direct factor of $A$ then $\widetilde{B} = B$ and if $B$ is not a direct factor of $A$ then $A/ B$ is not torsion-free,  $A \simeq A_1 \times A_2$, where $A_1$, $A_2$ are finitely generated abelian groups, $A_2$ could be trivial, $B \leq A_1 = \sqrt{B}$, $A_1/ B$ is finite of order $p^k m$, where $m$ is coprime to $p$, $k \geq 1$. Note that $k \geq 1$ by the choice of $p$ at the beginning of the proof. In the latter
case $\widetilde{B} \leq A_1$ and $|A_1/ \widetilde{B} | = p^k$. In all cases
 $\widetilde{B} \not=  A$. 
 
 Finally by \cite[Remark 2]{Ellis2}  
 \begin{equation} \label{eq-eq3} \cap_m[ \Gamma_{n+1}(F), _m F]  = Ker (\Gamma_{n+1}(F) \to \gamma_{n+1}(F))  = A/ B\end{equation}
 On the other hand 
$$\cap_m[ \Gamma_{n+1}(F), _m F]  =\cap_m [(F \wedge \gamma_n(F))/B , _m F] = (\cap_m [(F \wedge \gamma_n(F)) , _m F]B)/ B \leq A/B$$
hence by (\ref{eq-eq2})
$$\cap_m [(F \wedge \gamma_n(F)) , _m F]B = (\cap_m [(F \wedge \gamma_n(F)) , _m F]B) \cap A
  \subseteq \widetilde{B} \not= A$$ a contradiction with (\ref{eq-eq3}).
\end{proof}

\section{One corollary of Theorem A}

In this section we denote by $\theta_{n, G}$ the map previously denoted $\theta_n: \Gamma_n(G) \to \gamma_n(G)$, since we will need to use this map for different groups $G$.

\begin{prop} \label{parafreeiso1}  Let $P$ be a  finitely generated parafree group of rank $r$, so  there is a free group $F$ of rank $r$ such that $P/ \gamma_n(P) \simeq F/ \gamma_n(F)$ for all $n$ with compatible isomorphisms.  Then for  $n\geq 2, m \geq 1$ 
 $$\Gamma_n(P)/ [\Gamma_n(P), _m P] \simeq \Gamma_n(F)/ [\Gamma_n(F), _m F]$$ 
 and the group epimorphism $\theta_{n,P} : \Gamma_n(P) \to \gamma_n(P)$ induces an isomorphism
$$\theta_{n,m,P} :  \Gamma_n(P)/ [\Gamma_n(P), _m P ] \to \gamma_n(P) / \gamma_{n+m} (P)$$ \end{prop}

\begin{proof}  We fix $n$ and $m$.
Let $\pi_{n+m} : P \to P / \gamma_{n+m}(P)$ be the canonical epimorphism. Consider a group homomorphism
$$\alpha : F \to P \hbox{ such that } \pi_{n+m} \alpha : F \to P/ \gamma_{n+m}(P) \hbox{ has kernel } \gamma_{n+m}(F)$$
so we get an isomorphsim
$$\widetilde{\alpha} : F/ \gamma_{n+m}(F) \to P / \gamma_{n+m}(P)$$
Note that $n,m$ are fixed now and the map $\alpha$ depends on $n+m$ i.e. there is no $\alpha$ that works for all possible values of $n+ m$. 

The map $\alpha$ induces a homomorphism of multiplicative Lie algebras $L(F) \to L(P)$, that itself induces a group homomorphism $\Gamma_n(F) \to \Gamma_n(P)$, that induces a group homomorphism $$\alpha_{n,m} : \Gamma_n(F)/ [\Gamma_n(F), _m F ] \to \Gamma_n(P)/ [\Gamma_n(P), _m P ] $$
There is a group epimorphism $\theta_{n,P} : \Gamma_n(P) \to \gamma_n(P)$ and this induces an epimorphism
$$\theta_{n,m,P} :  \Gamma_n(P)/ [\Gamma_n(P), _m P ] \to \gamma_n(P) / \gamma_{n+m} (P)$$
By Theorem A the map $\theta_{n, F} : \Gamma_n(F) \to \gamma_n(F)$ is an isomorphism  and induces an isomorphism of groups 
$$\theta_{n,m,F}  : \Gamma_n(F)/ [\Gamma_n(F), _m F] \to \gamma_n(F)/ \gamma_{n+m}(F)$$
Consider the group homomorphism
$$\theta_{n,m,P}  \alpha_{n,m} \theta_{n,m,F}^{-1}: \gamma_n(F)/ \gamma_{n+m}(F) \to  \gamma_n(P) / \gamma_{n+m} (P)$$
that is a restriction of the isomorphism $\widetilde{\alpha}$, hence $\theta_{n,m,P}  \alpha_{n,m} \theta_{n,m,F}^{-1}$ is an isomorphism.
Hence $\theta_{n,m,P}  \alpha_{n,m}$ is an isomorphism. Hence $\alpha_{n,m}$ is injective.

The fact that $\alpha_{n,m}$ is surjective follows from the following claim applied for $G = P$. Indeed by construction $Im(\alpha) \gamma_{n+m}(P) = P$, hence $Im (\alpha) \gamma_{m+1}(P) = P$ and we note that an arbitrary element $a_i \in P$ decomposes as $a_i = g_i h_i$ where $g_i \in Im (\alpha)$, $h_i \in\gamma_{m+1}(P)$. This implies that $\alpha_{n,m}$ is an isomorphism and consequently $\theta_{n,m}$ is an isomorphism.

\medskip
{\bf Claim} {\it Let $G$ be a group and  $ g_1, \ldots, g_n \in G$, $h_1, \ldots, h_n \in \gamma_{m+1}(G)$. Then
$$
\{ g_1 h_1, \ldots , g_i h_i,  \ldots, g_n h_n \} \in \{ g_1, \ldots, g_i, \ldots, g_n \} [\Gamma_n(G), _m G]
$$}

\medskip Proof. It suffices to show it for a finitely generated free group $F_0$ and then consider a homomorphism $F_0 \to G$ such that the image contains $h_1, \ldots, h_n, g_1, \ldots, g_n$. This induces a homomorphism of multiplicative Lie algebras $\mu : L(F_0) \to L(G)$ that sends $[\Gamma_n(F_0), _m F_0]$ inside $[\Gamma_n(G), _m G]$. Note that by definition $\mu$ commutes with $\{, \}$-bracket and with commutator bracket $[ ,]$.

Now suppose $G = F_0$. By Theorem A $\theta_{n, F_0} : \Gamma_n(F_0) \to \gamma_n(F_0)$ is an isomorphism. Let
$$c := \{ g_1 , \ldots, g_i , \ldots, g_n  \}^{-1} \{ g_1 h_1, \ldots, g_i h_i, \ldots, g_n h_n \}$$
Then
$$b: =\theta_{n, F_0}(c) =[ g_1, \ldots, g_i, \ldots, g_n ]^{-1} [ g_1 h_1, \ldots, g_i h_i,  \ldots, g_n h_n ] $$
Since the Lie algebra ( over $\Z$) associated to the lower central series $\gamma_n (F)$ of $F$ is the free Lie algebra over the free abelian group $F_0^{ab} = F_0/ F_0'$ we get that the image of $b$ in $F_0/ \gamma_{m+n}(F_0)$ is zero. Thus $b \in \gamma_{m+n}(F_0)$.
But $\theta_{n, F_0}^{-1} (\gamma_{m+n}(F_0)) = \theta_{n, F_0}^{-1} ( [\gamma_n(F_0), _m F_0]) = [\Gamma_n(F_0), _m F_0]$, hence $$c = \theta_{n, F_0}^{-1} (b) \in [\Gamma_n(F_0),_m F_0]$$
 \end{proof}

\begin{cor} \label{EllisCondition}   
 Let $P$ be a  finitely generated parafree group. Then $\cap_m [\Gamma_n(P), _m P] = Ker (\Gamma_n(P) \to \gamma_n(P))$.
\end{cor}

\begin{proof} By Proposition \ref{parafreeiso1} $\theta_{n,m, P}$ is an isomorphism.
Hence  $[\Gamma_n(P), _m P ]$ is the full preimage $ \theta_{n, P}^{-1}( \gamma_{n+m} (P))$.

 Since $P$ is parafree, it  is residually nilpotent, i.e.  $\cap_m \gamma_{n+m}(P) = 1$.
Then  $ \cap_m [\Gamma_n(P), _m P ] = \cap_m  \theta_{n, P}^{-1}( \gamma_{n+m} (P)) = Ker ( \theta_{n, P})$.
\end{proof}

\section{Proof of Theorem B}

We list below the properties of a finitely generated free group $F$ that were used in the proof of Theorem A. The conditions where the prime $p$ appears should hold for any prime $p$.

1) $H_3(F/ R, \mathbb{Z}) \simeq Ker ( F \wedge R \to [F,R])$ where $R$ is a normal subgroup of $F$, used in Lemma \ref{L3};

2) For the pro-$p$ group $\widehat{F}_p$ we have $H_2^{cont} ( \widehat{F}_p, \mathbb{Z}_p ) = 0 = H_3^{cont}( \widehat{F}_p, \mathbb{Z}_p )$ ( in the free case $ \widehat{F}_p$   has pro-$p$ cohomological dimension $1$),  used in Lemma \ref{L1};

3) $H_3(G, \mathbb{Z})$ is finitely generated for $G = F / \gamma_n(F)$ used in Lemma \ref{L2}  and in the final stage of proof of the Main Theorem where we used that $A$ is finitely generated abelian;

3)' the fact that $H_3(G, \mathbb{Z})$ is torsion-free is used in Lemma \ref{L2} but we have stated Lemma \ref{L2} as a separate fact in 4) thus we can ignore the torsion-free condition;

4) the map $H_3(G, \mathbb{Z}) \to H_3(\widehat{G}_p, \mathbb{Z}_p)$ is injective for $G = F / \gamma_n(F)$ and every prime $p \geq 2$, this is Lemma \ref{L2};

5) $F$ is residually $p$-finite, used in the final part of Lemma \ref{L3};

6) $\cap_{m \geq 1} [\Gamma_n(F),_m F] = Ker (\Gamma_n(F) \to \gamma_n(F))$, this is \cite[Remark 2]{Ellis2}, used in the last paragraph of the proof of  Theorem \ref{Main}.

For a pro-$p$ group $B$ we have that it is free pro-$p$ if $H^2_{cont}(B, \mathbb{F}_p) = 0$, that by Pontryagin duality is equivalent to $H_2^{cont}(B, \mathbb{F}_p) = 0$. This is equivalent to $H_2^{cont}(B, \Z_p) = 0$. 

We note that by \cite{B-L1} if $H_2(F, \mathbb{Z}) = 0 = H_3(F, \mathbb{Z})$ then 
 condition 1) from the above list holds. Thus we obtain the following result, substituting $F$ with $P$ and substituring $\overline{\gamma_n}$ with the closure of $\gamma_n(P)$ in the pro-$p$ completion  $\widehat{P}_p$ i.e. the smallest pro-$p$ subgroup of  $\widehat{P}_p$ that contains $\gamma_n(P)$.
 
 \begin{theorem} \label{cor} Suppose that $P$ is a group such that for all $n \leq k$:
 
 \smallskip
 a) $H_2(P, \mathbb{Z}) = 0 = H_3(P, \mathbb{Z})$ and $H_3(P/ \gamma_n(P), \mathbb{Z})$ is finitely generated,
 
 \smallskip
 b) $\cap_{m \geq 1} [\Gamma_n(P),_m P] = Ker (\Gamma_n(P) \to \gamma_n(P))$,
 
 \smallskip
 c) for every prime integer $p \geq 2$ 
 
 \smallskip 
 c1) the map $H_3(G, \mathbb{Z}) \to H_3(\widehat{G}_p, \mathbb{Z}_p)$ is injective for $G = P / \gamma_n(P)$,
 
 c2) $P$ is residually $p$-finite,

c3) $\widehat{P}_p$ is  a free pro-$p$ group.

\medskip
Then the map $\theta_i: \Gamma_i(P) \to \gamma_i(P)$ is an isomorphism for $i \leq k+1$.
\end{theorem} 

\iffalse{In view of the proof of the cases $n = 2,3$ in \cite{Ellis2} and \cite[Remark 2]{Ellis2} we believe that conditions a), b) and c2) are natural ones. Furthermore if $G$ is finitely generated and $H_3(G, \mathbb{Z})$ is torsion-free and $G$ is (homologically) $p$-good in  dimension 3 (so the isomorphism (\ref{p-good}) holds) then by Lemma \ref{BK1} c1) holds. }\fi

Next we consider the class of finitely generated parafree groups introduced by Baumslag  in \cite{Baumslag}.

\begin{prop} \label{parafree} Let $P$ be a finitely generated, parafree group such that $H_2(P, \mathbb{Z}) = 0 = H_3(P, \mathbb{Z}) $ (e.g. this holds if $P$ satisfies the Strong Parafree Conjecture). Then all assumptions of Theorem \ref{cor} hold. \end{prop}

\begin{proof} There is a finitely generated free group $F$ such that $P/ \gamma_n(P) \simeq F/ \gamma_n(F)$ for every positive $n$ and the isomorphisms are compatible in an obvious way. Hence conditions a) and c1) hold, where for c1) we use Lemma \ref{L2}. Condition b) is Corollary \ref{EllisCondition}.
By Proposition \ref{Jaikin} $P$ is residually $p$-finite and $\widehat{P}_p$ is a free pro-$p$ group for every prime $p$, hence condition c2) and c3)  hold. 
\end{proof}

\begin{theorem}  (= Theorem B)  Let $P$ be a parafree group such that $H_2(P, \mathbb{Z}) = 0 = H_3(P, \mathbb{Z}) $ ( e.g. $P$  satisfies the Strong Parafree Conjecture). Suppose further that if $P$ is not finitely generated then it is free.  Then $\theta_n : \Gamma_n(P) \to \gamma_n(P)$ is an isomorphism for every $n \geq 1$. \end{theorem}

\begin{proof}If $P$ is finitely generated then the result follows from Theorem \ref{cor} and Proposition \ref{parafree}.

If $P$ is not finitely generated by assumption $P = F$ is free.
Let $X$ be a fixed free basis of $F$ and $\{ X_i \}_{i \in I}$ all finite subsets of $X$. \iffalse{, with an order $(I, \leq)$ given by $i \leq j$ if $X_i \subseteq X_j$. Note that $I$ is a directed set i.e. for $i,j \in I$ there is $k \in I$ such that $i\leq k, j \leq k$}\fi Set $F_i$ the free ( discrete) group with a free basis $X_i$. Note that since $F_i$ is a retract of $F$ via the  epimorphism $F \to F_i$ that is the identity on $X_i$ and sends all element of $X \setminus X_i$ to 1. Then $\Gamma_n(F_i)$ is a retract of $\Gamma_n(F)$ and $\gamma_n(F_i)$ is a retract of $\gamma_n(F)$. Thus the map $F_i \to F$ induces injective homomorphisms $\Gamma_n(F_i) \to \Gamma_n(F)$ and $\gamma_n(F_i) \to \gamma_n(F)$. Identifying $\Gamma_n(F_i)$ with its image in $\Gamma_n(F)$ we get that $\Gamma_n(F) = \cup_i \Gamma_n(F_i)$, similarly $\gamma_n(F) = \cup_i \gamma_n(F_i)$. Hence the restriction of the map $ \theta_{n,F}: \Gamma_n(F) \to \gamma_n(F)$ to $\Gamma_n(F_i)$ is the map $\theta_{n,F_i} : \Gamma_n(F_i) \to \gamma_n(F_i)$ that by Theorem A is an isomorphism. Hence the map $\theta_n$ is an isomorphism.

\end{proof}

{\bf Proof of Corollary C} Let 
$P = \langle X \ | \ r \rangle$ be a presentation with one relator $r $, $|X| = d$. Since $P$ is parafree there is a free group $F$ such that $P/  \gamma_n(P) \simeq F / \gamma_n(F)$ for every $n$, with obvious compatibility conditions.

Since $P$ is residually nilpotent $\cap_n \gamma_n(P) = 1$. Since $P / \gamma_n(P) \simeq F / \gamma_n(F)$ is torsion-free we get that $P$ is torsion-free , hence $r$ is not a proper power. Hence by \cite{Lyn}  $cd(P) \leq 2$, hence $H_3(P, \Z) = 0$.

Let $F_0$ be the free group with free basis $X$ and let $R$ be the normal subgroup of $F_0$ generated by $r$. If $r \in F_0'$ then $P/ \gamma_2(P) \simeq F_0/ \gamma_2(F_0)$ is free abelian of rank $|X|$, hence the free group $F$ has rank $d$ and by a result of Magnus \cite{Mag}  $P$ is free of rank $|X|$, so $r = 1$.

Suppose $r \not\in F_0'$. Then $R/ (R \cap F_0') \leq F_0/ F_0'$, so $R/ (R \cap F_0')$ is a non-trivial  subgroup in the free abelian group $F_0/ F_0'$, hence is torsion-free.
Since $[R,F_0] \leq R \cap F_0' \leq R$ and $R/ [R, F_0]$ is cyclic group with torsion-free non-trivial quotient $R/ (R \cap F_0')$ we conclude that $[R,F_0] = R \cap F_0'$ and $R/ [R, F_0] \simeq \Z$. Hence by the Hopf formula $H_2(P, \Z) \simeq (R \cap F_0') / [R,F_0] = 0$.

\end{document}